\documentclass[11pt]{article}
\usepackage[a4paper,top=3cm,bottom=2cm,left=3cm,right=3cm,marginparwidth=2cm]{geometry}

\usepackage{amsmath}
\usepackage{amssymb}
\usepackage{multirow}
\usepackage{mathtools}
\usepackage{amsthm}
\usepackage{graphicx}
\usepackage[colorinlistoftodos]{todonotes}
\usepackage[colorlinks=true, allcolors=blue]{hyperref}

\usepackage{pgf,tikz}
\usepackage{mathrsfs}
\usepackage[symbol]{footmisc}

\usetikzlibrary{arrows}

\usepackage[english]{babel}
\usepackage[utf8x]{inputenc}
\usepackage[T1]{fontenc}

\newtheorem{definition}{Definition}[section]
\newtheorem{theorem}[definition]{Theorem}
\newtheorem{lemma}[definition]{Lemma}

\newtheorem{remark}[definition]{Remark}

\newtheorem{prop}[definition]{Proposition}

\renewenvironment{proof}[1][\noindent Proof]{{\par\pushQED{\qed}\itshape #1\@. }}{\popQED}

\newcommand{\cS}{\mathcal{S}}
\newcommand{\F}{\mathbb{F}}
\newcommand{\V}{\mathbb{V}}
\newcommand\spanset[1]{\ensuremath \langle #1 \rangle}

\title{On sets of subspaces with two intersection dimensions and a geometrical junta bound}
\author{Giovanni Longobardi\footnote{Corresponding Author} , Leo Storme and Rocco Trombetti}
\date{}

\begin{document}

\maketitle

\begin{abstract} In this article, constant dimension subspace codes whose codewords have subspace distance in a prescribed set of integers, are considered. The easiest example of such an object is a {\it junta} \cite{junta}; i.e. a subspace code in which all codewords go through a common subspace.

We focus on the case when only two intersection values for the codewords, are assigned. In such a case we determine an upper bound for the dimension of the vector space spanned by the elements of a non-junta code. In addition, if the two intersection values are consecutive, we prove that such a bound is tight, and classify the examples attaining the largest possible dimension as one of four infinite families.  

\end{abstract}

\section{Introduction and preliminaries}

Let $\V=V(\F)$ be a finite dimensional vector space over a (possibly finite) field $\F$, and let $k \in \mathbb{Z}^+$ be a positive integer and $\ell \in \mathbb{N}$, such that $\ell < k$. A $(k;\ell)$-{\it SCID} (Subspaces with Constant Intersection Dimension) in $\V$, is a set of $k$-dimensional subspaces of $\V$ ($k$-spaces in the following) pairwise intersecting in an $\ell$-dimensional space \cite{barrolleta-storme}. The easiest way of constructing such an object is by considering a so-called $\ell$-$sunflower$. Precisely, by taking a family $\mathcal{S}$ of $k$-spaces of $\V$ containing an $\ell$-space $V'$ and having no points in common outside of $\V'$. In the following, we will refer to $V'$ as the {\it center} of the sunflower, and will call the elements of $\cS$, the {\it petals} of $\cS$.

Of course, not all $\ell$-sunflowers with the same number of petals span a subspace of the same dimension in $\mathbb{V}$. A sunflower $\cS$ is said to be of \textit{maximal dimension} if among all sunflowers with the same number of petals, it spans a subspace of $\mathbb{V}$ of largest dimension. 

In this article we focus on a natural generalization of the concept of $(k;\ell)$-SCID. More precisely, let $\ell_1,\ell_2,...,\ell_v \in \mathbb{N}$ be non-negative integers such that $\ell_1,\ell_2, ..., \ell_v < k$. We give the following definition.

\begin{definition}\label{definitionfamily}
A set $\mathcal{S}$ of $k$-spaces of $\V$ is a $(k;\ell_1,\ell_2,...,\ell_v)$-SPID (Subspaces with Pre-assigned Intersection Dimensions) if for each pair of distinct subspaces $\pi_i,\pi_j\in \cS$, we have $ \dim(\pi_i \cap \pi_j) \in \{\ell_1,\ell_2,...,\ell_v\}$, and for each integer $\ell_m \in \{\ell_1,\ell_2,...,\ell_v\},$ there exist at least two $k$-spaces in $\mathcal{S}$ such that $ \dim(\pi_i \cap \pi_j)=\ell_m$.
\end{definition} 

\noindent For our purposes, we always suppose that the dimension of $\V$ to be large enough in order to assure the existence of such an object. Clearly, in the case when $v=1$, we get back the definition of a $(k;\ell)$-SCID in $\V$. 

The notion of $\ell$-sunflower in $\V$ can also be naturally generalized. We say that a $(k;\ell_1,\ell_2,...,\ell_v)$-SPID $\mathcal{S}$ is an {\it $\ell$-junta} in $\V$, if all elements of $\mathcal{S}$ pass through a common $\ell$-space of $\V$.

These geometric objects arise from a more general problem stated in \cite{beutelspacher,deboeck, eisfeld}, and recently gained a particular interest due to the fact that they provide {\it constant subspace codes}, which are a main tool in random network coding \cite{Etion,KK,subspacecodes}.

\medskip
\noindent In this paper, we elaborate on finite $(k;\ell_1,\ell_2,...,\ell_v)$-SPIDs, mainly focusing on the case when only two intersection values are assigned. Moreover, if $\tilde{\ell} = \min\{\ell_1,\ell_2,...,\ell_v\}$, we will assume that $\dim(\V) \geq |\mathcal{S}|(k-\tilde{\ell})+\tilde{\ell}$. Under this hypothesis, it is easy to see that $\cS$ is an $\tilde{\ell}$-sunflower of $\mathbb{V}$ of maximal dimension if any element $\pi \in \mathcal{S}$ meets the subspace generated by all others precisely in the center. In this case, we determine an upper bound for the dimension of the vector space spanned by the elements of a non-junta code, providing the smallest intersection value is strictly larger than zero. In addition, if these two possible intersection values are consecutive integers, we prove that this bound is tight and classify the examples attaining the largest dimension as one of four infinite families.  

\medskip

\noindent Let $\mathcal{S}=\{\pi_{1},\pi_{2},\ldots,\pi_{n}\}$ be a  $(k;\ell_1,\ell_2,\ldots,\ell_v)$-SPID. As in \cite{barrolleta-storme}, for each $j \in \{1,...,n\}$, the differences 

\[ \delta_{j}(\cS)=\dim\spanset{\pi_{1},\ldots,\pi_{j}}-\dim\spanset{ \pi_{1},\ldots,\pi_{j-1}} \] 
will be an important arithmetic tool in order to prove our results. We underline here that we consider the span of the empty set as the null subspace; accordingly we put $\delta_1=k$. Clearly, the values $\delta_j(\cS)$ depend on the labeling of the subspaces in $\mathcal{S}$. In the following, we will enclose these integers in an array, say $\delta(\mathcal{S})=(\delta_1(\mathcal{S}),\ldots,\delta_n(\mathcal{S}))$. Regarding this array, we show the following fact which also will play a crucial role. 

\begin{prop}\label{sorting}
Let $k,t_1,t_2,\ldots,t_v \in \mathbb{Z}^+$ be integers such that $k\geq t_1 > t_2 > \cdots >t_v \geq 1$. 
Let $\cS=\{\pi_1,\ldots,\pi_n\}$ be a $(k;k-t_1,k-t_2,\ldots,k-t_v)$-SPID in a vector space $\V$, with $n \geq 3$. Then there exists a permutation $\sigma$ of the indices in the set $I_n=\{1,2,\dots,n\}$ such that
\begin{equation*}
    t _1= \delta_2(\cS_\sigma)\geq \delta_3(\cS_\sigma) \geq \ldots \geq \delta_n(\cS_\sigma),
\end{equation*}
where $\cS_{\sigma}=\{\pi_{\sigma(1)},\ldots,\pi_{\sigma(n)}\}$, and
\begin{equation*}
    \delta_j(\cS_\sigma)=\dim\spanset{ \pi_{\sigma(1)},\ldots,\pi_{\sigma(j)}}-\dim\spanset{ \pi_{\sigma(1)},\ldots,\pi_{\sigma(j-1)}}.
\end{equation*}
\end{prop}
\begin{proof}
Let $m \in \mathbb{Z}^+$ be the maximum integer for which there exist $m$ $k$-spaces, $\pi_{i_1},\pi_{i_2},\ldots,\pi_{i_m}$, of $\mathcal{S}$, forming a $(k-t_1)$-sunflower of maximal dimension; obviously $m \geq 2$.\\
Consider
\begin{equation*}
    \underset{i \not= i_1,\ldots,i_m}{\max_{1\leq i \leq n }} \dim(\pi_i \cap \spanset{\pi_h \,|\,h\not=i}),
\end{equation*}
and denote by $i_n$ an integer in $\{1,\ldots,n\}\setminus \{i_1,\ldots,i_m\}$ such that
\begin{equation*}
    \dim (\pi_{i_n} \cap \spanset{\pi_h \,|\,h\not=i_n})= \underset{i \not= i_1,\ldots,i_m}{\max_{1\leq i \leq n }} \dim(\pi_i \cap \spanset{\pi_h \,|\,h\not=i}).
\end{equation*}
Similarly, consider
\begin{equation*}
    \underset{i \not= i_1,\ldots,i_m,i_n}{\max_{1\leq i \leq n }} \dim(\pi_i \cap \spanset{\pi_h \,|\,h\not=i,i_n}),
\end{equation*}
then there exists an integer, say $i_{n-1} \in \{1,\ldots,n\}\setminus \{i_1,\ldots,i_m,i_n\}$, such that
\begin{equation*}
    \dim (\pi_{i_{n-1}} \cap \spanset{\pi_h \,|\,h\not=i_{n-1}, i_n})= \underset{i \not= i_1,\ldots,i_m,i_n}{\max_{1\leq i \leq n }} \dim(\pi_i \cap \spanset{\pi_h \,|\,h\not=i,i_n}).
\end{equation*}
After $n-m$ steps, we obtain a sequence of indices $(i_{m+1},\ldots,i_n)$.

Let $\sigma$ be a permutation of the indices $\{1,\ldots,n\}$, fixing the set $\{i_1,i_2,\ldots,i_m\}$ and such that $\sigma(j)=i_j$, for every $j=m+1,\ldots,n$.
Consider $\cS_\sigma= \{ \pi_{\sigma(1)},\ldots,\pi_{\sigma(n)}\}$, we will show that

\[ \delta_{j+1}(\cS_\sigma) \leq \delta_j(\cS_\sigma) \,\,\,\, \textnormal {for all }\,\, j=2,\ldots,n-1.\]
First of all, $\delta_j(\cS_\sigma)\leq t_1$, for each $j=2,\ldots,n$; indeed $$\delta_j(\cS_\sigma)=k-\dim(\pi_{\sigma(j)} \cap \spanset{\pi_{\sigma(1)},\ldots,\pi_{\sigma(j-1)}}) \leq k - \dim(\pi_{\sigma(j)} \cap \pi_{\sigma(1)}) \leq t_1.$$ Also, since $\pi_{\sigma(1)},\ldots,\pi_{\sigma(m)}$ form a $(k-t_1)$-sunflower of maximal dimension, then we have $\delta_j(\cS_\sigma)=t_1$, with $2\leq j \leq m$.\\
Note that 
\begin{equation*}
\dim(\pi_{i_{j+1}} \cap \spanset{ \pi_h \,|\,h\not = i_{j+1},\ldots,i_n})  \geq \dim(\pi_{i_{j}} \cap \spanset{ \pi_h \,|\,h\not = i_{j},\ldots,i_n})
\end{equation*}
for all $m+1 \leq j \leq n-1$, because otherwise we would have 
\begin{equation*}
    \begin{split}
  \dim(\pi_{i_{j+1}} \cap \spanset{ \pi_h \,|\,h\not = i_{j+1},\ldots,i_n}) <  \dim(\pi_{i_{j}} \cap \spanset{ \pi_h \,|\,h\not = i_{j},\ldots,i_n}) & \leq \\
  & \hspace{-4cm} \dim(\pi_{i_j} \cap   \spanset{\pi_h \,|\, h \not=i_j,i_{j+2},\ldots,i_n} ),
    \end{split}
\end{equation*}
which is a contradiction by the definition of $i_{j+1}$. Then
\begin{equation*}
    \delta_{j+1}(\cS_\sigma)=k-\dim(\pi_{i_{j+1}} \cap \spanset{ \pi_h \,|\,h\not = i_{j+1},\ldots,i_n})  \leq k- \dim(\pi_{i_{j}} \cap \spanset{ \pi_h \,|\,h\not = i_{j},\ldots,i_n})=\delta_j(\cS_\sigma).
\end{equation*}
This concludes the proof.
\end{proof}\\

In other terms, it is always possible to sort $k$-spaces in $\mathcal{S}$ in such a way that the associated array $\delta(\mathcal{S})$, is non-increasing (see also \cite[Theorem 2]{barrolleta-storme}). 
 
\begin{remark}\label{rm:parametersforms}
{\rm  We note explicitly that for a $(k;k-t)$-SCID $\cS=\{\pi_1,\ldots,\pi_n\} \subset \V,$ we have  
\begin{equation}\label{sunflowersequence}
    \delta(\mathcal{S}):= (\delta_1,\delta_2,\ldots,\delta_n)=(k,t,\ldots,t)
 \end{equation}
if and only if $\mathcal{S}$ is a $(k-t)$-sunflower of maximal dimension. The necessary condition is in fact trivial. While, regarding the sufficiency we may observe that since $\delta_n(\cS)=t$ and \[\pi_n \cap \spanset{\pi_1,\ldots,\pi_{n-1}} \supseteq \pi_n \cap \pi_i,\] we get for each $i \in \{1,...,n-1\},$ $\pi_n \cap \pi_i= V'$, where $V'=\pi_n \cap \spanset{\pi_1,\ldots,\pi_{n-1}}$. This implies that $\cS$ is a $(k-t)$-sunflower. Finally, by using Grassmann's formula, it is easy to show that $\cS$ is of maximal dimension.}
\end{remark}

\section{A junta-property bound for $(k;k-t_1,k-t_2)$-SPIDs}

In this section, we restrict our discussion to the case where only two values for the intersection dimensions are possible.

We start by showing a result which appears as a quite natural generalization of \cite[Theorem 2]{barrolleta-storme} to  $(k;k-t_1,k-t_2)$-SPID, with $k-t_1 \neq 0$. 

\begin{theorem} \label{initial}
Let $k,t_1,t_2 \in \mathbb{Z}^+$ such that $k> t_1 > t_2 \geq 2$. Let $\mathcal{S}$ be a $(k;k-t_1,k-t_2)$-SPID in $\V$, with $n=| \cS | \geq 3$. If $\dim \spanset{\cS}  \geq k+(t_1-1)(n-1)+2$, then $\cS$ is a $(k-t_1)$-junta.
\end{theorem}
\begin{proof}
Let $\delta(\mathcal{S})$ be any non-increasing array associated with $\cS$. In particular, arguing as in the proof of Proposition \ref{sorting}, we can choose as first $m$ spaces, $m \geq 2$, those forming a $(k-t_1)$-sunflower of maximal dimension. By Remark \ref{rm:parametersforms}, the integer $m$ is the largest index for which $\delta_m=t_1$. Let $V'$ be the center of the sunflower formed by $\pi_1,\ldots,\pi_m$. Hence, we get $\dim V'=k-t_1$.\\
Assume that $\cS$ is not a $(k-t_1)$-junta, so we can find a subspace $\pi_r \in \cS \setminus \{\pi_1,\ldots,\pi_m\}$ not containing $V'$. 
We denote $k-t_1-\dim(\pi_r \cap V')$ by $\varepsilon$; hence, $\varepsilon \geq 1$. Also, in the quotient vector space $\Pi=\spanset{\cS}/(V'\cap \pi_r)$, we have that $\dim_{\Pi}\pi_r=t_1+\varepsilon$, and that $\dim_{\Pi}(\pi_r \cap \pi_i) \in \{\varepsilon,\varepsilon+t_1-t_2\}$, for each $1 \leq i \leq m$. Also, the subspaces $(\pi_r \cap \pi_i)/(V' \cap \pi_r)$ of $\Pi$, with $1 \leq i \leq m$, are linearly independent.
Hence,
\begin{equation*} 
\begin{split}
\delta_{r} & = \dim \pi_r -\dim(\langle \pi_1, \pi_2, \ldots,\pi_{r-1} \rangle \cap \pi_r)\leq \dim_\Pi \pi_r - \dim_\Pi \spanset{\pi_1 \cap \pi_r, \ldots, \pi_m \cap \pi_r }  \\ 
 & \hspace{1cm}= t_1+\varepsilon - \sum^m_{i=1}  \dim_{\Pi}(\pi_r \cap \pi_i) \leq t_1+\varepsilon - m \cdot \varepsilon \leq t_1-m+1.\end{split}
\end{equation*}
Since $\delta(\cS)=(\delta_1,\ldots, \delta_n)$ is non-increasing, it is easy to see that
\begin{equation} \label{bound}
\begin{split}
\dim \spanset{\mathcal S} & = \sum^n_{i=1} \delta_i= k + \sum^m_{i=2} \delta_i +\sum^{r-1}_{i=m+1}\delta_i+\sum^n_{i=r} \delta_i \\
&  \leq k + (m-1)t_1+(r-m-1)(t_1-1)+(n-r+1)(t_1-m+1) \\
& =k+(n-1)(t_1-1)-(n-r)(m-2)+1 \\
& \leq k+(n-1)(t_1-1)+1,
\end{split}
\end{equation}
which proves the theorem.
\end{proof}

\begin{remark}\label{lower_bound}
\textnormal{We point out here that unlike what happens for SCIDs, in general the bound stated above is not tight. For instance, with the same notation as used in Theorem \ref{initial}; if $t_1 > t_2 \geq 2$ and there exists an integer $s$ such that $r>s>m$ with $\delta_s \leq t_2$, we can slightly improve on the upper bound stated in Theorem \ref{initial}. In fact, if this is the case we can repeat the proof of Theorem \ref{initial}, and by re-writing Inequality (\ref{bound}), we get}
\begin{equation}\label{lower_bound_counting}
\begin{split}
\dim \spanset{\mathcal S} & = \sum^n_{i=1}\delta_i= k + \sum^m_{i=2}\delta_i +\sum^{s-1}_{i=m+1}\delta_i+\sum^{r-1}_{i=s}\delta_i+\sum^n_{i=r}\delta_i \\
&  \leq k + (m-1)t_1+(s-m-1)(t_1-1)+(r-s)t_2+(n-r+1)(t_1-m+1) \\
& =k+(n-1)(t_1-1)-(n-r)(m-2)-(r-s)(t_1-t_2-1)+1 \\
& \leq k+(n-1)(t_1-1)-(t_1-t_2)+2.
\end{split}
\end{equation}

\textnormal{This possibility can be realised if the first $r-1$ spaces form a $(k-t_1)$-junta with $\dim(\pi_s \cap \pi_j)=k-t_2$ for some $j\in\{1,\ldots,s-1\}$. In what follows, we exhibit a concrete example.} 

\textnormal{Let $k,t_1,t_2 \in \mathbb{Z}^+$ such that $k> t_1 > t_2+1 > 2$  and consider $t_1-t_2+1 \leq m \leq \min\{t_1+1,n-1\}$. Let $V',X,N_1,\ldots,N_m$, $M_{m+1},\ldots,M_{s-1}$ and $P_s, \ldots,P_{n-1}$ be linearly independent subspaces of $\V$ such that}
\begin{itemize}
	\item [$a)$] $\dim V'=k-t_1$,
	\item [$b)$] $\dim X=t_1-m+1$,
	\item [$c)$] $\dim N_i=t_1$,   \, \textnormal{for} $i=1,\ldots,m$,
	\item [$d)$] $\dim M_j = t_1-1$,  \, \textnormal{for} $j=m+1,\ldots,s-1$,
	\item [$e)$] $\dim P_\ell=t_2$, \, \textnormal{for} $\ell=s,\ldots,n-1$.
\end{itemize}
\textnormal{Let $A_i=\{a_{i1},\ldots,a_{i,t_1-t_2}\}$ be a set of linearly independent $1$-spaces in $N_i$, for $i=1,\ldots,m$, $|A_i|=t_1-t_2$, and we choose in $A_i$ a 1-space, for example $a_{i1}$. Now, let $b_{m+1},\ldots, b_{s-1}$ be distinct $1$-spaces in $\spanset{a_{11},\ldots,a_{m1}}\setminus\{a_{11},\ldots,a_{m1}\}$ (where $\frac{q^m-1}{q-1}\geq s-m-1$ when $\mathbb{V}$ is a vector space over the Galois field of order $q$, $\F_q$) and let $W$ be a $(k-t_1-1)$-space in $V'$. Then we define the $k$-spaces $\pi_1,\ldots,\pi_n$ as follows.}
\begin{itemize}
	\item [$\circ$] $\pi_1=\spanset{V',N_1}$, $\pi_2=\spanset{V',N_2}$, $\ldots$,  $\pi_m=\spanset{V',N_m}$,
	\item[$\circ$] $\pi_{m+1}=\spanset{V',b_{m+1}, M_{m+1}}$, $\ldots$, $\pi_{s-1}=\spanset{V',b_{s-1},M_{s-1}}$,
	\item [$\circ$] $\pi_s=\spanset{V',Q_s,P_s}$,$\ldots$,$\pi_{n-1}=\spanset{V',Q_{n-1},P_{n-1}}$,
	\item [$\circ$] $\pi_{n}=\spanset{W,a_{11},\ldots,a_{m1},X}$,
	\end{itemize}
\textnormal{where $Q_s,\ldots,Q_{n-1}$ are $(t_1-t_2)$-spaces equal to $\spanset{A_i}$, for some $i \in\{1,\ldots,m\}$.\\
It is easy to verify that
\begin{itemize}
	\item [$i)$] $\pi_i \cap \pi_j =V'$, with $i,j=1,\ldots,s-1,$
	\item[$ii)$] $\pi_i \cap \pi_j=V'$, with $i=m+1,\ldots,s-1$ and $j=s,\ldots,n-1,$
	\item[$iii)$] $\dim(\pi_i \cap \pi_j) \in \{k-t_1,k-t_2\}$,  for $i=1,\ldots,m$ and $j=s,\ldots,n$.
\end{itemize}}

\begin{figure}[h!]
	\centering
	\definecolor{ccqqqq}{rgb}{0.8,0,0}
	\definecolor{qqwuqq}{rgb}{0,0.39215686274509803,0}
	\definecolor{qqqqff}{rgb}{0,0,1}
	\begin{tikzpicture}[line cap=round,line join=round,>=triangle 45,x=1cm,y=1cm,scale=0.52]
	\clip(-15.17,-7.324) rectangle (12.756,8.802);
	\draw [line width=1pt] (-7.29,3.85) circle (2.3654175107156026cm);
	\draw [rotate around={59.62087398863186:(-7.49,3.39)},line width=0.8pt,color=qqqqff] (-7.49,3.39) ellipse (1.3377342264228316cm and 1.1565175573864341cm);
	\draw [rotate around={-71.56505117707799:(5.012,0.176)},line width=0.8pt] (5.012,0.176) ellipse (7.728656901081648cm and 4.442087065179722cm);
	\draw [line width=0.8pt] (3.012,4.924) circle (1.9912126958213174cm);
	\draw [line width=0.8pt] (6.882,-4.684) circle (1.9449791772664304cm);
	\draw [rotate around={-54.34467190209972:(3.153,4.996)},line width=0.8pt,color=qqwuqq] (3.153,4.996) ellipse (1.567439308953724cm and 1.4384230904894868cm);
	\draw [rotate around={-30.32360686254994:(6.893,-4.761)},line width=0.8pt,color=qqwuqq] (6.893,-4.761) ellipse (1.3548597302177934cm and 1.1745104889126488cm);
	\draw [rotate around={-69.2308352052139:(6.244,0.805)},line width=0.8pt,color=ccqqqq] (6.244,0.805) ellipse (6.766777322054547cm and 1.5940044938054767cm);
	\draw [line width=0.8pt] (-10.96,-1.34) circle (1.7157179255343808cm);
	\draw [line width=0.8pt,dash pattern=on 3pt off 3pt on 1pt off 4pt] (-3.026,6.628) circle (1.1269498657881816cm);
	\draw [line width=0.8pt] (-3.062,-1.538) circle (1.8931508127986cm);
	\draw [line width=0.8pt,dash pattern=on 3pt off 3pt on 1pt off 4pt] (-12,6) circle (1.0605375995220532cm);
	\draw [rotate around={18.121860247901363:(-7.231,-5.388)},line width=0.8pt,dash pattern=on 2pt off 2pt] (-7.231,-5.388) ellipse (2.589586245635795cm and 1.748969960744354cm);
	\begin{scriptsize}
	\draw[color=black] (-9.09,6.305) node {$V'$};
	\draw[color=black](-6,4.633) node {$W$};
	
	\draw[color=black] (3.294,8.471) node {$\spanset{N_1,\ldots,N_m}$};
	\draw[color=black] (4.66,-2.957) node {$N_m$};
	\draw[color=black] (1.7,4) node {$Q_1$};
	\draw[color=black] (2.086,2.645) node {$N_1$};
	\draw[color=black] (7,-6.3) node {$Q_{m}$};
	\draw[color=black] (8.5,4.193) node {$\langle a_{11},\ldots,a_{m1} \rangle $};
	\draw [fill=black] (2.57,0.508) circle (1pt);
	\draw [fill=black] (3.758,-1.604) circle (1pt);
	\draw [fill=black] (2.966,-0.636) circle (1pt);
	\draw [fill=black] (2.284,4.6) circle (1.5pt);
	\draw[color=black] (3.2,4.2) node {$a_{1,t_1-t_2}$};
	\draw [fill=black] (3.934,5.392) circle (1.5pt);
	\draw[color=black] (4.264,4.9) node {$a_{11}$};
	\draw [fill=black] (7.454,-4.244) circle (1.5pt);
	\draw[color=black] (8.2,-4.321) node {$a_{m 1}$};
	\draw [fill=black] (6.178,-5.124) circle (1.5pt);
	\draw[color=black] (7.15,-5.6) node {$a_{m,t_1-t_2}$};
	\draw[color=black] (-11.73,1) node {$M_{m+1}$};
	\draw[color=black] (-3.26,8.185) node {$P_{n-1}$};
	\draw[color=black] (-2.93,1) node {$M_{s-1}$};
	\draw[color=black] (-12.434,7.723) node {$P_s$};
	\draw[color=black] (-10,-4.145) node {$X$};
	\draw [fill=black] (-9.156,7.328) circle (1pt);
	\draw [fill=black] (-8.012,7.592) circle (1pt);
	\draw [fill=black] (-6.802,7.526) circle (1pt);
	\draw [fill=black] (-8.078,-1.472) circle (1pt);
	\draw [fill=black] (-7.33,-1.538) circle (1pt);
	\draw [fill=black] (-6.604,-1.384) circle (1pt);
	\draw [fill=ccqqqq,color=ccqqqq] (5.012,2.994) circle (1.5pt);
	\draw [fill=ccqqqq,color=ccqqqq] (8.212,-2.7) circle (1.5pt);
	\draw [fill=ccqqqq,color=ccqqqq] (7.512,-0.5) circle (1.5pt);
	\draw [fill=ccqqqq,color=ccqqqq] (5.56,4) circle (1.5pt);
	\draw[color=black] (5.298,2.4) node {$b_{m+1}$};
	\draw [fill=ccqqqq, color=ccqqqq] (6.002,0.728) circle (1.5pt);
	
	\draw [fill=ccqqqq, color=ccqqqq] (7.146,-1.736) circle (1.5pt);
	\draw[color=black] (7.388,-1.285) node {$b_{s-1}$};
	\end{scriptsize}
	\end{tikzpicture}
	\caption{ The $(k;k-t_1,k-t_2)$-SPID described in Remark \ref{lower_bound}.}
\end{figure}
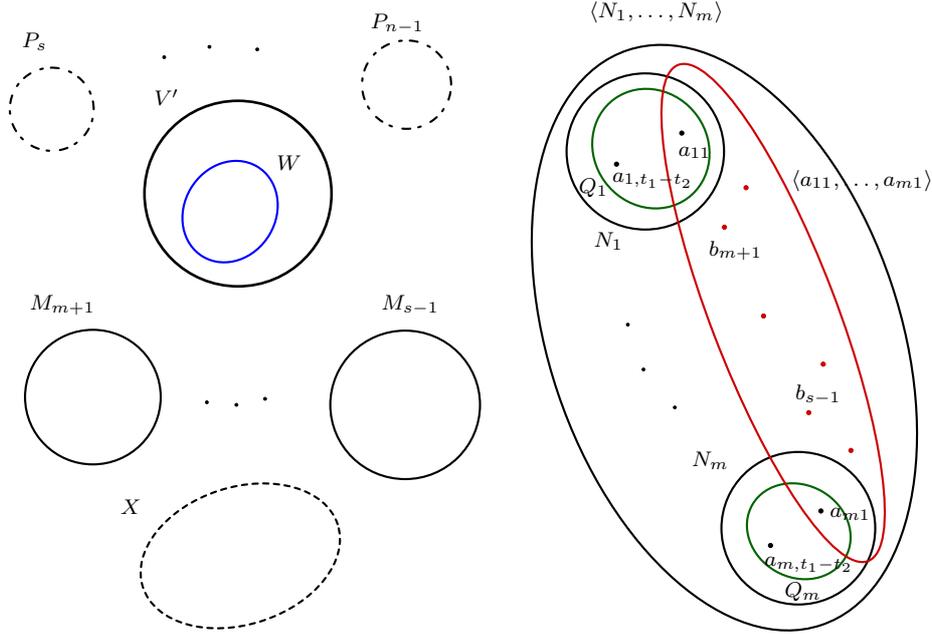

\textnormal{Hence, $\cS=\{\pi_1,\ldots,\pi_n\}$ is a set of $n$ distinct $k$-spaces pairwise meeting either in a space of dimension $k-t_1$ or of dimension $k-t_2$,  i.e. a $(k; k-t_1,k-t_2)$-SPID. Also, it is clear that $\cS$ is not a $(k-t_1)$-junta.\\
Now, we have that}
\begin{equation*}
\spanset{\cS}=\spanset{\pi_1,\ldots,\pi_n}=\spanset{V',N_1,\ldots,N_m,M_{m+1},\ldots,M_{s-1},P_s,\ldots,P_{n-1},X}.
\end{equation*}
\textnormal{So, by hypothesis,}
\begin{equation*}
\begin{split}
\dim\spanset{\cS}& = k +
(m-1)t_1+(s-m-1)(t_1-1)+(n-s)t_2+(t_1-m+1) \\
& =k+(n-1)(t_1-1)-(n-s)(t_1-t_2-1)+1 \\
& \leq k+(n-1)(t_1-1)-(t_1-t_2)+2.
\end{split}
\end{equation*}
\textnormal{We find that the array $\delta(\cS)$ corresponding to such a SPID is as follows:}
\begin{equation*}
\delta(\cS)=(k,\underbrace{t_1,\ldots, t_1}_{m-1 \,\,\text{times}}, \underbrace{t_1-1, \ldots, t_1-1}_{s-m-1\,\, \text{times}}, \underbrace{t_2, \ldots, t_2}_{n-s\,\, \text{times}}, t_1-m+1).
\end{equation*} 
\end{remark}

\medskip

\noindent In the following, we will show that if in addition we ask that the two possible values for the dimensions of the intersection between elements of the SPID are consecutive integers, then the bound in Theorem \ref{initial} is sharp. Towards this aim, we put beforehand the following result.

 \begin{prop}\label{turningspaces}
Let $t_1,t_2 \in \mathbb{Z}^+$ with $k > t_1 > t_2 \geq 2$. Let $\cS$ be a $(k;k-t_1,k-t_2)$-SPID in a vector space $\V$, with $n=|\cS| \geq 3$, such that $\dim\spanset{\cS}=k+(n-1)(t_1-1)+1$. Also let $\delta(\cS)$ be any non-increasing array associated with $\cS$.\\
Then, there is no $(k-t_1)$-sunflower of maximal dimension with at least three petals in $\cS$, if and only if
\begin{equation}\label{secondsequence}
\delta(\cS)=(k,t_1,t_1-1,\ldots,t_1-1).
\end{equation}
Moreover, if $\cS =\{\pi_1,\pi_2,\ldots,\pi_n\}$ has associated non-increasing array $\delta(\cS)$ like in (\ref{secondsequence}), then any permutation $\sigma$, fixing $\pi_1$ and $\pi_2$, does not change $\delta(\cS)$.
\end{prop}
\begin{proof}  The necessity is obvious because if any such a $\delta(\cS)$ is like in (\ref{secondsequence}), then, by Proposition \ref{sorting} and by Remark \ref{rm:parametersforms}, $\cS$ can not contain a  $(k-t_1)$-sunflower of maximal dimension with at least three petals. 

Regarding sufficiency, clearly we have $\dim(\pi_1 \cap\pi_2)=k-t_1$ and by hypothesis in any non-increasing array the largest index $m$ for which $\delta_m=t_1$, is 2. Now, if $\delta_n \leq t_1-2$, then
\begin{equation*}
\dim \spanset{\cS}=\sum^n_{i=1}\delta_i=k+t_1+\sum^{n-1}_{i=3} \delta_i+\delta_n \leq k+ t_1 +(n-3)(t_1-1)+t_1-2= k+(n-1)(t_1-1),
\end{equation*}
a contradiction. Hence, $(\delta_1,\delta_2,\ldots,\delta_n)=(k,t_1,t_1-1,\ldots,t_1-1)$.\\
Now, we show that any permutation of the $k$-spaces in $\cS$, fixing $\pi_1$ and $\pi_2$, does not change the array $(\ref{secondsequence})$.
First of all, we notice that
\begin{equation*}
  \hspace{3cm}  \dim( \pi_j \cap \spanset{\pi_1,\pi_2}) = k-t_1+1, \hspace{2cm} \textnormal{for all} \,\,j=3,\ldots,n.
\end{equation*}
Indeed, for $3 \leq j \leq n$,
  \begin{equation*}
      k-t_1 \leq \dim(\pi_j \cap \pi_1) \leq \dim( \pi_j \cap \spanset{\pi_1, \pi_2}) \leq \dim( \pi_j \cap \spanset{\pi_1,\pi_2,\ldots,\pi_{j-1}})=k-t_1+1.
  \end{equation*}
  If $\dim(\pi_j \cap \spanset{\pi_1,\pi_2})=k-t_1$, then
  \begin{equation*}
  \pi_j \cap \pi_1=\pi_j \cap \spanset{\pi_1 ,\pi_2}=\pi_j \cap \pi_2.
  \end{equation*} 
Consequently, we also have $\pi_j \cap \spanset{\pi_1 ,\pi_2}=\pi_1 \cap \pi_2$. This implies that $\pi_1,\pi_2,\pi_j$ form a $(k-t_1)$-sunflower of maximal dimension; in fact, $\dim\spanset{\pi_1,\pi_2,\pi_j}=k+2t_1$ and, eventually applying the same procedure as in the proof of Proposition \ref{sorting}, we would get a non-increasing array $\delta(\cS)$ with $\delta_2=\delta_3=t_1$; a contradiction. Hence, $\dim( \pi_j \cap \spanset{\pi_1,\pi_2}) = k-t_1+1$.

Nevertheless, since any non-increasing array $\delta(\cS)$ is as in \eqref{secondsequence}, we have
\begin{equation*}
    k-t_1+1=\max_{3 \leq i \leq n} \dim (\pi_i \cap \spanset{\pi_h \,|\, h \not = i}) \geq \dim (\pi_j \cap \spanset{\pi_h \,| \,h \not = j}) \geq \dim(\pi_j \cap \spanset{\pi_1,\pi_2})=k-t_1+1.
\end{equation*}
Hence, for any $I \subset I_n=\{1,\ldots,n\}$, with $1,2\in I$ and  $j \not \in I$,
\begin{equation*}
\hspace{3cm} \dim (\pi_j \cap \spanset{\pi_h \,| \,h \in I})=k-t_1+1.
\end{equation*}

Now, let $\sigma$ be any permutation of $I_n$ such that $\sigma(1)=1$ and $\sigma(2)=2$, then 
\begin{equation*}
\begin{split}
    \delta_j(\cS_\sigma)= & k-\dim (\pi_{\sigma(j)} \cap \spanset{\pi_{\sigma(1)},\pi_{\sigma(2)}, \ldots,\pi_{\sigma(j-1)}})=\\
   & k-\dim (\pi_{\sigma(j)} \cap \spanset{\pi_{1},\pi_{2},\pi_{\sigma(3)}, \ldots,\pi_{\sigma(j-1)}})=t_1-1,
    \end{split}
\end{equation*}
for all $j=3,\ldots,n$.

\end{proof}

\medskip
\noindent Next, we exhibit four families of $(k;k-t,k-t+1)$-SPIDs which are not ($k-t$)-juntas, and such that $\dim\spanset{\cS} = k+ (n-1)(t-1)+1$.


In the following, we will denote by $\delta'(\cS)$ any non-increasing array obtained as described in Proposition \ref{sorting}.

\subsection{SPIDs with $\delta'=(k,t,\dots,t,t-1,\dots,t-1,t+1-m)$}

\noindent Let $t \in \mathbb{Z}^+$ such that $2 \leq t \leq k-1$. Let $m \in \mathbb{Z}^+$ be a positive integer such that $m>2$. We provide a class of $(k;k-t,k-t+1)$-SPIDs with non-increasing array $$ \delta'(\cS) = (k,\underbrace{t, \ldots, t}_{m-1 \,\,\text{times}}, \underbrace{t-1, \ldots, t-1}_{n-m-1\,\, \text{times}}, t+1-m).$$

\medskip
\noindent {\bf $\bullet$ \, Class I}
\vspace{0.2cm}

\textnormal{Let $2<m \leq \min\{t+1,n-1\}$. Let $V',X,N_1,\ldots,N_m$ and $M_{m+1},\ldots,M_{n-1}$ be linearly independent subspaces of $\V$ such that $\dim V'=k-t$, $\dim X=t-m+1$, $\dim N_i=t$ for $i=1,\ldots,m,$ and $\dim M_j = t-1$ for $j=m+1,\ldots,n-1$ (Figure \ref{figexample1}).\\
Let $a_1,\ldots,a_m$ be $1$-spaces in $N_1,\ldots,N_m$, respectively. Also, let $b_{m+1},\ldots, b_{n-1}$ be $1$-spaces in $\spanset{a_1,\ldots,a_m}$ such that either
\begin{itemize}
    \item [$(i)$] at least two of them are the same $1$-space, or
    \item [$(ii)$] at least one of them is equal to $a_i$, with $i \in \{1,\ldots,m\}$.
\end{itemize}
Let $W$ be a $(k-t-1)$-space in $V'$. Then we define the $k$-spaces $\pi_1,\ldots,\pi_n$ as follows.
\begin{itemize}
    \item [$\circ$] $\pi_1=\spanset{V',N_1}$, $\pi_2=\spanset{V',N_2}$, $\ldots$,  $\pi_m=\spanset{V',N_m}$,
    \item[$\circ$] $\pi_{m+1}=\spanset{V',M_{m+1},b_{m+1}}$, $\ldots$, $\pi_{n-1}=\spanset{V',M_{n-1},b_{n-1}}$,
    \item [$\circ$] $\pi_{n}=\spanset{W,a_1,\ldots,a_m,X}$.
\end{itemize}
By Requests $(i)$ and $(ii)$, it is clear that the pairwise intersection of distinct spaces $\pi_i$ and $\pi_j$, $i,j=1, \ldots, n-1$, either is the $(k-t)$-space $V'$ or it is a $(k-t+1)$-space containing $V'$. Moreover, since each of the spaces $\pi_1,\ldots,\pi_{n-1}$ contains a unique $1$-space from the set $\{a_1,\ldots,a_m,b_{m+1},\ldots,b_{n-1}\}$ (note that by Properties $(i)$ and $(ii)$, some of the $1$-spaces could be equal), we have $\dim(\pi_n \cap \pi_i)=k-t$, for all $i=1,\ldots,n-1$.\\
Hence, the set $\cS=\{\pi_1,\ldots,\pi_n\}$ is a set of $n$ distinct $k$-spaces pairwise meeting in a space of dimension $k-t$ or $k-t+1$. Also, since not all pairwise intersections equal the same $(k-t)$-space, $\cS$ is not a $(k-t)$-junta.\\
The set $\{a_1,\ldots,a_m,b_{m+1},\ldots,b_{n-1}\}$ is contained in $\spanset{N_1,\ldots,N_m}$ and $W \subset V'$. Then
\begin{equation*}
\spanset{\cS}=\spanset{\pi_1,\ldots,\pi_n}=\spanset{V',N_1,\ldots,N_m,M_{m+1},\ldots,M_{n-1},X}.
\end{equation*}
Clearly, since $V',X,N_1,\ldots,N_m$ and $M_{m+1},\ldots,M_{n-1}$ are linearly independent spaces of $\V$, we have that
\begin{align*}
\dim\spanset{\cS}& = k-t+m \cdot t+ (n-1-m) \cdot (t-1) + t-m+1 \\
& = k+ (n-1)(t-1)+1.
\end{align*}
Arguing as in Proposition \ref{sorting}, we find
\begin{equation*}
    \delta'(\cS)=(k,\underbrace{t, \ldots, t}_{m-1 \,\,\text{times}}, \underbrace{t-1, \ldots, t-1}_{n-m-1\,\, \text{times}}, t+1-m).
\end{equation*} }

\begin{figure}[h!]
	\centering
\definecolor{qqqqff}{rgb}{0,0,1}
	\definecolor{ccqqqq}{rgb}{0.8,0,0}
	\begin{tikzpicture}[line cap=round,line join=round,>=triangle 45,x=1cm,y=1cm,scale=0.35]
	\clip(-24.449808129530336,-13.420236848002814) rectangle (28.599380115097805,16.698721959608733);
	\draw [line width=0.8pt] (-13.568449754941666,7.417275259221697) circle (4.7456545844381cm);
	\draw [color=qqqqff,line width=0.8pt] (-13.965900873626893,8.530138391540335) circle (3.040262058333512cm);
	\draw [rotate around={-49.74860106050114:(2.9632897790603066,6.649205701979444)},line width=0.8pt] (2.9632897790603066,6.649205701979444) ellipse (6.7253285394534705cm and 5.971865410278013cm);
	\draw [color=ccqqqq,line width=0.8pt] (2.8884112653388643,6.124366770138653) circle (3.3322298235126295cm);
	\draw [line width=0.8pt] (5.14871577730176,9.154136647876143) circle (1.8375768964829389cm);
	\draw [line width=0.8pt] (4.138792484722597,2.806047380235687) circle (1.837576896482939cm);
	\draw [line width=0.8pt] (-0.6222744660077462,7.326656404161465) circle (1.8375768964829404cm);
	\draw [line width=0.8pt,dash pattern=on 2pt off 2pt] (-15.29020800108607,-6.283262252976783) circle (3.8584319807040073cm);
	\draw [line width=0.8pt,dash pattern=on 1pt off 3pt on 3pt off 4pt] (0.43574041193233,-6.860361277307734) circle (2.8762627980932627cm);
	\draw [line width=0.8pt,dash pattern=on 1pt off 3pt on 3pt off 4pt] (11.304438703498564,-4.407690423901193) circle (2.8762627980932605cm);
	\begin{scriptsize}
	\draw[color=black] (-13.387645242532884,13.731705322103952) node {$V'$};
	\draw[color=black] (-13.886303500937048,4.4) node {$W$};
	\draw[color=black] (2.7688823297620613,14.33009523218895) node {$\spanset{N_1,\ldots,N_m}$};
	\draw[color=black] (7.107209177878296,7.39874544037106) node {$N_2$};
		\draw[color=black] (9.107209177878296,5.39874544037106) node {$\spanset{a_1,\ldots,a_m}$};
	\draw[color=black] (3.1179431106449766,0.8663222552765036) node {$N_m$};
	\draw[color=black] (-2.4171635576412545,9.141365861593966) node {$N_1$};
	\draw[color=black] (-15.332412450309127,-11.5) node {$X$};
	\draw[color=black] (0.2755910377412365,-10.5) node {$M_{m+1}$};
	\draw[color=black] (11.196206896792448,-8) node {$M_{n-1}$};
	\draw [fill=black] (4.475433582248984,8.240396526018804) circle (2pt);
	\draw[color=black] (5.012844492580804,8.196598653817725) node {$a_2$};
	\draw [fill=black] (3.754059801835296,3.7197875020930256) circle (2pt);
	\draw[color=black] (4.165125453293723,4.27) node {$a_m$};
	\draw [fill=black] (0.3876488265714181,6.941923721274166) circle (2pt);
	\draw[color=black] (0.7742492961454014,7.5) node {$a_1$};
	\draw [color=ccqqqq,fill=ccqqqq] (2.407495411729745,7.470931160244204) circle (2pt);
	\draw[color=black] (2.3,8.1) node {$b_{n-1}$};
	\draw [color=ccqqqq,fill=ccqqqq] (4.407495411729745,6) circle (2pt);
	\draw [color=ccqqqq,fill=ccqqqq] (1.541846875233319,4.970168721476751) circle (2pt);
	\draw[color=black] (1.921163290474981,5.85290483931815) node {$b_{m+1}$};
	\draw[color=ccqqqq,fill=ccqqqq] (4.165125453293723,4.27)  (5.8,6.6) circle (2pt);
	\draw[fill=black] (4.165125453293723,4.27) (5,-7) circle (1pt);
	\draw[fill=black] (4.165125453293723,4.27)  (5.5,-6.8) circle (1pt);
		\draw[fill=black] (4.165125453293723,4.27)  (6,-6.5) circle (1pt);
	
	\end{scriptsize}
	\end{tikzpicture}
	\caption{The ($k;k-t,k-t+1$)-SPID described in Class I.}
	\label{figexample1}
\end{figure}
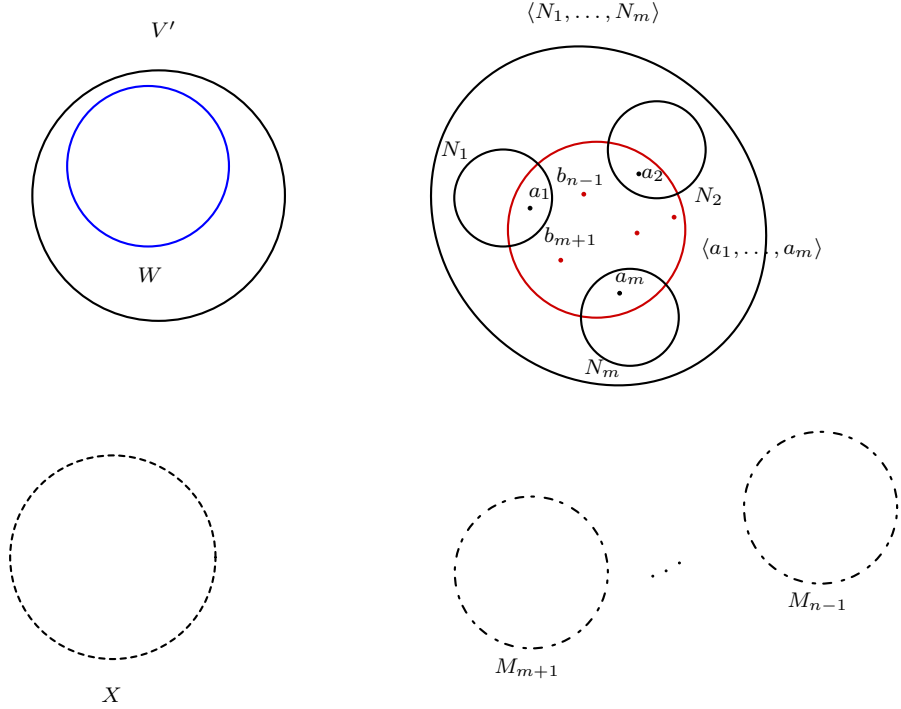

\begin{lemma} \label{3petals}
Let $\cS$ be a $(k;k-t, k-t+1)$-SPID of $\V$, where $2 \leq  t \leq k-1$, such that $\cS$ is not a $(k-t)$-junta, with $| \cS |= n  \geq 3$.\\
If $\dim\spanset{\cS}=k+(n-1)(t-1)+1$  and there exists a $(k-t)$-sunflower of maximal dimension with at least three petals in $\cS$, then $\cS$ belongs to Class I. 
\end{lemma}

\begin{proof}
Since $\dim\spanset{\cS}=k+(n-1)(t-1)+1$, from the proof of Theorem \ref{initial} we get
\[
\sum^n_{i=1} \delta_i=k+(n-1)(t_1-1)-(n-r)(m-2)+1.\] Moreover, since this implies that $(n-r)(m-2)=0$ and $m \geq 3$, necessarily $r=n$. Then,
\begin{equation} \label{case1}
    \delta'(\cS)=(k,\underbrace{t, \ldots, t}_{m-1 \,\,\text{times}}, \underbrace{t-1, \ldots, t-1}_{n-m-1\,\, \text{times}}, t+1-m).
\end{equation}

Consider $\cS'=\{\pi_1,\ldots,\pi_{n-1}\}$. Since $\dim \spanset{\cS'}=k+(n-2)(t-1)+m-1 \geq k+(n-2)(t-1)+2$, then, by Theorem \ref{initial}, we have that $\cS'$ is a $(k-t)$-junta.\\
Let $V'$ be the common $(k-t)$-space through which the $k$-spaces $\pi_1, \ldots, \pi_{n-1}$ pass, and denote $k-t-\dim(\pi_n \cap V')$ by $\varepsilon$. Since $\cS$ is not a junta, $\varepsilon \geq 1$; indeed, by the proof of Theorem \ref{initial},  necessarily $\varepsilon=1$. Let $W$ denote the $(k-t-1)$-subspace $ \pi_n \cap V'$.\\
Furthermore, we note the first $k$-spaces form a sunflower of maximal dimension since $\delta_2=\cdots=\delta_m=t$. Hence, there exist $t$-spaces $N_1, \ldots,N_m$, with $i=1,\ldots,m$, such that $N_1,\ldots,N_m,V'$ are linearly independent, and $\pi_i=\langle V', N_i \rangle$. \\
Also, by hypothesis, there exist at least two $k$-spaces in $\cS$ such that they meet in a $(k-t+1)$-space.\\
We first show that
\begin{equation*}
   \hspace{3cm} \dim(\pi_n \cap \pi_j)=k-t, \hspace{2cm} \textnormal{for all} \,\, j \in \{1,\ldots,n-1\}.
\end{equation*} 
For this purpose, suppose by way of contradiction that there exists a $j \in \{1,\ldots,n-1\}$ such that $\dim(\pi_n \cap \pi_j)=k-t+1$; we may distinguish between two cases:
\begin{itemize}
    \item [$(a)$] \textit{ $j \in \{1, \ldots, m\}$}. Then there are two $1$-spaces $a_{j_1}$ and $a_{j_2}$ in $ \pi_n \cap \pi_j$ not in $V'$, and there is at least another $1$-space $a_i \in \pi_n \cap \pi_i$, for all $i \in \{1,\ldots,m\} \setminus \{j\}$ not in $V'$.  Without loosing any generality, we may choose the $N_i$'s in such a way that $a_1 \in N_1, \ldots, a_m \in N_m$ and $\langle  a_{j_1}, a_{j_2} \rangle \subseteq  N_j$. Hence,
    \begin{equation*}
        \pi_n \cap \spanset{ \pi_1, \ldots, \pi_{n-1}} \supseteq \spanset{ W, a_1,\ldots,a_{j-1},a_{j_1},a_{j_2},a_{j+1},\ldots,a_{m}},
    \end{equation*}
    obtaining that
    \begin{equation*}
     t-m+1=\delta_n \leq k- \dim \spanset{W, a_1,\ldots,a_{j-1},a_{j_1},a_{j_2},a_{j+1},\ldots,a_{m}}=t-m,
    \end{equation*}
 a contradiction.  
 \item[$(b)$] $j \in \{m+1, \ldots, n-1\}$. Since, by Point $(a)$, $\dim(\pi_n \cap \pi_i)=k-t$ for every $i=1,\ldots,m$,  $\pi_n$ contains the $1$-spaces $a_1 \in \pi_1,\ldots, a_m  \in \pi_m$, meeting $V'$ trivially. Furthermore, since $\dim(\pi_n \cap \pi_j)=k-t+1$, there must be two $1$-spaces $a',a'' \in \pi_n \cap \pi_j$ not in $V'$ and such that $\langle a', a'' \rangle \cap W$ is trivial.
Also, the subspace $\spanset{a',a''}$ can not be contained in $\spanset{V',a_1,\ldots,a_m}$, otherwise we would have
 \begin{equation*}
     \pi_{j} \cap \spanset{ V', a_1,\ldots, a_m}  \supseteq \langle V',a',a'' \rangle,
 \end{equation*}
 and, consequently,
\begin{equation*}
t-1=\delta_{j}\leq k- \dim \spanset{V',a',a''}=t-2.
\end{equation*}
Moreover, $\spanset{a',a''}$ meets $\spanset{V',a_1,\ldots,a_m}$ in a 1-space, otherwise
\begin{equation*}
    \pi_n \cap \spanset{\pi_1,\ldots,\pi_{n-1}} \supseteq \spanset{W,a_1,\ldots,a_m,a',a''}
\end{equation*}
obtaining again $t-m+1=\delta_n \leq t-m-1$.
However, if $b \in \spanset{a',a''} \setminus \spanset{V',a_1,\ldots,a_m}$, then
\begin{equation*}
t-m+1=\delta_n \leq k- \dim \spanset{W,a_1\ldots,a_m,b}=t-m;
\end{equation*}
which is again a contradiction.
\end{itemize}
Hence, definitely $\dim(\pi_n \cap \pi_j)=k-t$, for each $j \in \{1,\ldots,n-1\}$. 

\medskip

\noindent Now, since $\delta_n=t-m+1$ and $\pi_n$ intersects $V'$ in the $(k-t-1)$-dimensional subspace $W$, we get that the $k$-space $\pi_n$ may be realised as follows $$\pi_n=\spanset{ W,a_1,\ldots,a_m,X},$$ for suitable points $a_1 \in N_1,\ldots, a_m\in N_m$ and $X$ a $(t-m+1)$-dimensional subspace such that $V',N_1,\ldots,N_m,X$ are linearly independent.

Since $\pi_n \cap \pi_j$, $j=m+1,\ldots,n-1$,  is a $(k-t)$-space contained in $\spanset{W,a_1,\ldots,a_m},$ there must  exist a 1-space $b_j$ in $\spanset{a_1,\ldots,a_m} \setminus W$ such that $\pi_n \cap \pi_j=\spanset{W,b_j}$ otherwise $$t-m+1=\delta_n=k-\dim(\pi_n \cap \spanset{\pi_1,\ldots,\pi_{n-1}})\leq k-\dim\spanset{W,a_1,\ldots,a_m,b_j}=t-m,$$
a contradiction. Moreover, since $\delta_j=t-1$, it is immediate that for any  $j=m+1,\ldots,n-1$, we have $\pi_{j}= \langle V',b_j,M_j\rangle$, with $M_j$ a $(t-1)$-space and such that $V',N_1,\ldots, N_m, M_{m+1},\ldots,$\\ $M_{n-1}$ and $X$ are linearly independent.\\
Note explicitly that if $\dim(\pi_i \cap \pi_j)=k-t+1 $, with $i,j \in \{m+1,\ldots,n-1\}$, then $b_i=b_j$. Indeed, let $\pi_i \cap \pi_j=\spanset{V',a'}$. This space is contained in  $\spanset{\pi_1,\pi_2,\ldots,\pi_m}$, since if $a' \not \in \spanset{\pi_1, \pi_2,\ldots,\pi_m}$, assuming $j >i$, we have
\begin{equation*}
    \pi_j \cap \spanset{\pi_1,\ldots,\pi_i} \supseteq \spanset{V',b_j,a'},
\end{equation*}
obtaining $\delta_j \leq t-2$. So, $\spanset{V',a'} \subseteq \spanset{\pi_1,\ldots,\pi_m}$.\\
Now, since $\pi_i \cap \spanset{\pi_1,\pi_2,\ldots, \pi_m} =\spanset{V',b_i}$ and  $\pi_j \cap  \spanset{\pi_1,\pi_2,\ldots,\pi_m}  =\spanset{V',b_j}$ have dimension $k-t+1$ and
\begin{equation*}
    \pi_i \cap \pi_j= \pi_i \cap \pi_j \cap \spanset{\pi_1, \pi_2, \ldots, \pi_m},
\end{equation*}
we get that both $\spanset{V',b_i}$ and $\spanset{V',b_j}$ are equal to $\pi_i \cap \pi_j$. Now, since we assumed $j > i$, if $b_i \not = b_j$ then again we would have
$t-1=\delta_j \leq t-2$, which is not the case.\\
Suppose that there exist $i \in \{1,\ldots,m\}$ and $j \in \{m+1, \ldots,n-1\}$ such that $\dim(\pi_i \cap \pi_j)=k-t+1$, then there exists a $1$-space $a' \in N_i$ such that
\begin{equation*}
    \spanset{V',a'}=\pi_i \cap \pi_j \subseteq \pi_j \cap \spanset{\pi_1,\ldots,\pi_m}=\spanset{V',b_j}.
\end{equation*}
Hence, $b_j \in \spanset{V',a'}$ and, since $\delta_j=t-1$, $b_j \in \spanset{a_1,\ldots,a_m} \cap N_i$ otherwise $$t-1=\delta_j=k-\dim(\pi_j \cap \spanset{\pi_1,\ldots,\pi_{j-1}}) \leq k -\dim\spanset{V',a',b_j}=t-2;$$
this implies that $b_j=a_i$.\\
Note explicitly that a $k$-space $\pi_j$ in $\cS$, with $j \in \{m+1,\ldots,n-1\}$, can meet at most one $\pi_i$, with $i \in \{1,\ldots,m\}$, in a $(k-t+1)$-space.\\
Finally, it is possible that in $\cS$ there exists a $k$-space $\pi_j$, with $j \in \{m+1,\ldots,n-1\}$, that intersects $\pi_i$, with $i \in \{1,\ldots,m\}$, and $\pi_h$, with $h\in \{m+1,\ldots,n-1\}$, in two $(k-t+1)$-spaces. From previous results, $b_h=b_j=a_i$. So, $\cS$ belongs to Class I. 
\end{proof}

\subsection{SPIDs with $\delta'=(k,t,t-1,t-1,\ldots,t-1)$}

\medskip
\noindent {\bf $\bullet$ \, Class II}
\vspace{0.2cm}

\noindent Choose integers $n \geq 3$ and $k,t$ such that $2 \leq t \leq k-1$. Let $W$ be a $(k-t+1)$-subspace of $\V$, and $X_1,X_2$ $t$-spaces such that $\dim \spanset{X_1,X_2}=2t-1$. Moreover, consider $M_3,\ldots, M_n$  $(t-1)$-subspaces of $\mathbb{V}$ such that $W$, $\spanset{X_1,X_2}$, $M_3,\ldots,M_n$ are linearly independent.
Let $W_1$  and $W_2$ be two $(k-t)$-spaces in $W$ (Figure \ref{fig:example2}).\\
Then we define the sets $\pi_1,\ldots,\pi_n$ as follows:
\begin{itemize}
    \item [$\circ$] $\pi_1=\spanset{W_1,X_1}$, $\pi_2=\spanset{W_2,X_2}$,
    \item[$\circ$] $\pi_{3}=\spanset{W,M_{3}}$, \ldots$,\pi_{n}=\spanset{W,M_n}$.
\end{itemize}
Now, since $\dim(X_1 \cap X_2)=1$, $\pi_1 \cap \pi_2$ is a $(k-t)$-space. Moreover, these two spaces meet other ones either in $W_1$ or in $W_2$, and $\{\pi_3,\ldots,\pi_n\}$ is a $(k-t+1)$-sunflower with center $W$.
Clearly,
\begin{equation*}
\spanset{\cS}=\spanset{\pi_1,\ldots,\pi_n}=\spanset{W,X_1,X_2,M_3,\ldots,M_{n}}.
\end{equation*}
Since $W$, $\spanset{X_1,X_2}$, $M_3,\ldots,M_n$ are linearly independent, we find that
\begin{align*}
\dim\spanset{\cS}& = k-t+1 +2t-1+(n-2)\cdot(t-1) \\
& = k+ (n-1)(t-1)+1.
\end{align*}
Again arguing as in Proposition \ref{sorting}, we get
\begin{equation*}
    \delta'(\cS)=(k,t,t-1, \ldots, t-1).
\end{equation*}

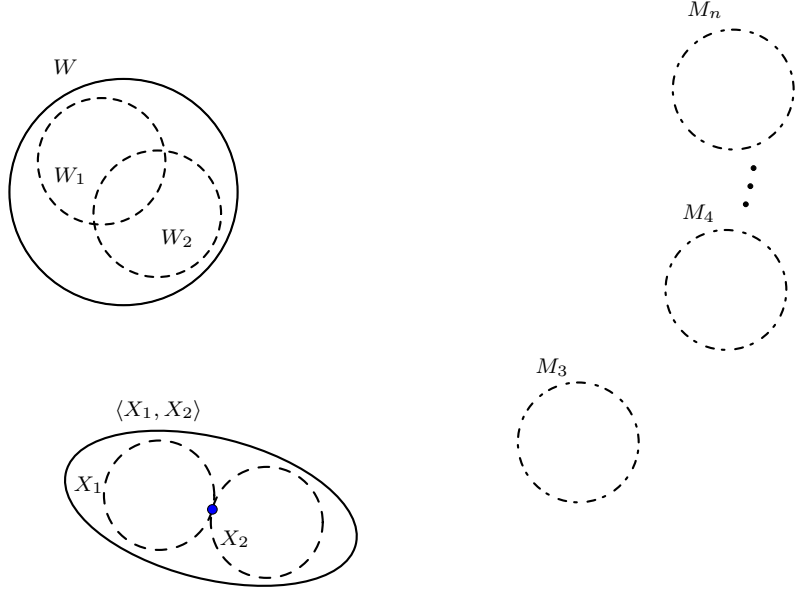
\begin{figure}[h!]
	\centering
	
	\begin{tikzpicture}[line cap=round,line join=round,>=triangle 45,x=1cm,y=1cm, scale=0.6]
	\clip(-12.5,-8.564248782363794) rectangle (15.281224732444986,6);
	\draw [line width=0.1pt] (-5.54,0.97) circle (2.500499950009998cm);
	\draw [line width=0.8pt] (-5.54,0.97) circle (2.500499950009998cm);
	\draw [line width=0.8pt,dash pattern=on 3pt off 3pt] (-4.8,0.49) circle (1.3984276885130673cm);
	\draw [line width=0.8pt,dash pattern=on 3pt off 3pt] (-6.02,1.65) circle (1.3978555004005235cm);
	\draw [line width=0.8pt,dash pattern=on 4pt off 4pt] (-4.76,-5.73) circle (1.2088010589009264cm);
	\draw [line width=0.8pt,dash pattern=on 4pt off 4 pt] (-2.4,-6.33) circle (1.2264828432820187cm);
	\draw [rotate around={-14.264512298079884:(-3.6247355639503485,-6.018626551538047)},line width=0.8pt] (-3.6247355639503485,-6.018626551538047) ellipse (3.274686538053768cm and 1.5602648844560156cm);
	\draw [line width=0.8pt,dash pattern=on 3pt off 3pt on 1pt off 4pt] (4.427766919525248,-4.562851209901641) circle (1.3242356285797476cm);
	\draw [line width=0.8pt,dash pattern=on 3pt off 3pt on 1pt off 4pt] (7.66457854811803,-1.1922892740590156) circle (1.3242356285797476cm);
	\draw [line width=0.8pt,dash pattern=on 3pt off 3pt on 1pt off 4pt] (7.8241289994439285,3.235252871129182) circle (1.3242356285797476cm);
	\begin{scriptsize}
	
	\draw[color=black] (-6.815269008050914,3.7441194176634207) node {$W$};
	\draw[color=black] (-4.366019099465826,-0.0463948806711543) node {$W_2$};
	\draw[color=black] (-6.708779881590693,1.330365884565077) node {$W_1$};
	\draw[color=black] (-6.3,-5.5) node {$X_1$};
	\draw [fill=blue] (-3.592479786071121,-6.043203687825456) circle (3pt);
	\draw[color=black] (-3.1,-6.69181497543824) node {$X_2$};
	\draw[color=black] (-4.774227417563341,-3.8521049364990136) node {$\spanset{X_1,X_2}$};
	\draw[color=black] (3.8513918257145763,-2.9) node {$M_3$};
	\draw[color=black] (7.063813807264583,0.5) node {$M_4$};
	\draw[color=black] (7.205799309211544,5) node {$M_n$};
	\draw [fill=black] (8.1,0.7) circle (1.5pt);
	\draw [fill=black] (8.2,1.1) circle (1.5pt);
	\draw [fill=black] (8.27,1.5) circle (1.5pt);
	
	\end{scriptsize}
	\end{tikzpicture}
	\caption{The $(k;k-t,k-t+1)$-SPIDs described in Class II.}
	\label{fig:example2}
\end{figure}

We observe that particular examples in this class contain $(k-t+1)$-sunflowers of maximal dimension, but do not contain  $(k-t)$-sunflowers of maximal dimension. Nonetheless, they are $(k-t-1)$-juntas.

\bigskip

\noindent {\bf $\bullet$ \, Class III} 
\vspace{0.2cm}

\noindent Choose integers $n\geq 3$, $2\leq s<n$ and $k,t$ such that $2 \leq t \leq k-1$. Let $\mathbb{V}$ be a vector space over a field $\mathbb{F}$ which is either infinite or else a finite field $\mathbb{F}$ of order $q$ with $q$ a prime power such that $q+1 \geq s$. Let $V'$, $\spanset{X_1,X_2},M_3,\ldots,M_n$ be linearly independent subspaces of $\mathbb{V}$ such that $\dim V'=k-t+2$, and $\dim X_1=t$, $\dim X_2=t-1$ with $\dim(X_1 \cap X_2)=1$ and $\dim M_i=t-1$, for $i=3,\ldots,n$.\\
Let $W_0,W_1,\ldots,W_s$ be  distinct $(k-t+1)$-spaces in $V'$ such that $W_1,\ldots,W_s$  go through a $(k-t)$-space $W$ (Figure \ref{fig:example3}), and $W_0$ does not pass through $W$. We define the sets

$$\pi_1  = \spanset{W,X_1}, \, \pi_2=\spanset{W_0,X_2},$$ 
$$\pi_{3}=\spanset{W_1,M_{3}}, \ldots,\, \pi_{m_1}=\spanset{W_1,M_{m_1}},$$ 
$$\pi_{m_1+1}=\spanset{W_2,M_{m_1+1}}, \ldots,\,  \pi_{m_2}=\spanset{W_2,M_{m_2}},$$
    $$ \ldots$$
$$\pi_{m_{s-1}+1}=\spanset{W_s,M_{m_{s-1}+1}},\ldots,\, \pi_{n}=\spanset{W_s,M_n}.$$

Clearly, $\cS=\{\pi_1,\ldots,\pi_n\}$ is a $(k;k-t,k-t+1)$-SPID which is not a $(k-t)$-junta and
\begin{equation*}
\spanset{\cS}=\spanset{\pi_1,\ldots,\pi_n}=\spanset{V',X_1,X_2,M_3,\ldots,M_{n}}.
\end{equation*}
Since $V'$, $\langle X_1,X_2 \rangle,M_3,\ldots,M_n$ are linearly independent, we find that
\begin{align*}
\dim\spanset{\cS}& = k-t+2+ 2t-2 +(n-2)(t-1) \\
& = k+ (n-1)(t-1)+1.
\end{align*}
Also in this case we have $\delta'(\cS)=(k,t,t-1, \ldots, t-1)$.

\begin{figure}[h!]
	\centering
	\begin{tikzpicture}[line cap=round,line join=round,>=triangle 45,x=1cm,y=1cm,scale=0.5]
	\clip(-12.225287442595883,-9.377756433558412) rectangle (16.875795269099516,5.897185128969249);

	\draw [line width=0.8pt] (-5.54,0.97) circle (3.745043714296057cm);
	\draw [line width=0.8pt,dash pattern=on 3pt off 3pt] (-4.76,-5.73) circle (1.2088010589009264cm);
	\draw [line width=0.8pt,dash pattern=on 3pt off 3pt] (-2.4,-6.33) circle (1.2264828432820187cm);
	\draw [rotate around={-14.264512298079884:(-3.6247355639503485,-6.018626551538047)},line width=0.8pt] (-3.6247355639503485,-6.018626551538047) ellipse (3.274686538053768cm and 1.5602648844560156cm);
	\draw [line width=0.8pt,dash pattern=on 3pt off 3pt on 1pt off 4pt] (4.427766919525248,-4.562851209901641) circle (1.2924395537122813cm);
	\draw [line width=0.8pt,dash pattern=on 3pt off 3pt on 1pt off 4pt] (7.66457854811803,-1.1922892740590156) circle (1.3242356285797476cm);
	\draw [line width=0.8pt,dash pattern=on 3pt off 3pt on 1pt off 4pt] (7.8241289994439285,3.235252871129182) circle (1.3446189051177297cm);
	\draw [rotate around={60.49484859913861:(-5.057400164165984,2.681701029239108)},line width=0.8pt,dash pattern=on 2pt off 2pt] (-5.057400164165984,2.681701029239108) ellipse (1.71107100606527cm and 1.1119699114429538cm);
	\draw [rotate around={-13.654785858284296:(-4.355486654152758,1.663926439719925)},line width=0.8pt,dash pattern=on 2pt off 2pt] (-4.355486654152758,1.663926439719925) ellipse (1.919070911979029cm and 1.0624198327160437cm);
	\draw [rotate around={-14.399604085272227:(-6.523522008206119,2.2579207475686203)},line width=0.8pt,dash pattern=on 2pt off 2pt] (-6.523522008206119,2.2579207475686203) ellipse (1.8922949304677805cm and 1.0457473852434787cm);
	\draw [rotate around={66.75709721404749:(-5.811957187430205,1.058526037333517)},line width=0.8pt,dash pattern=on 2pt off 2pt] (-5.811957187430205,1.058526037333517) ellipse (1.842516420798575cm and 1.08437568659382cm);
	\draw [line width=0.8pt] (-4.53973895053123,0.25132550081830646) circle (1.9075556995710221cm);
	\draw [line width=0.8pt,color=blue] (-5.4346786757980965,1.9359179248500535) circle (0.6986155426222715cm);
	\begin{scriptsize}
	\draw[color=black] (-7.349035421117334,4.7) node {$V'$};
	\draw[color=black] (-6.3,-5) node {$X_1$};
	\draw [fill=blue] (-3.592479786071121,-6.043203687825456) circle (3pt);
		\draw [color=black] (8,1.5) circle (1pt);
				\draw [color=black] (7.8,1) circle (1pt);
						\draw [color=black] (7.6,0.6) circle (1pt);
	\draw[color=black] (-3,-6.706080013579287) node {$X_2$};
	\draw[color=black] (-4.758622109847833,-3.8476929114888083) node {$\spanset{X_1,X_2}$};
	\draw[color=black] (3.8522690351997495,-2.8) node {$M_3$};
	\draw[color=black] (6.067954525051544,0) node {$M_4$};
	\draw[color=black] (5.8,4) node {$M_n$};
	\draw[color=black] (-5.8,4.2) node {$W_2$};
	\draw[color=black] (-3,3) node {$W_3$};
	\draw[color=black] (-7.241845904788941,3.61) node {$W_1$};
	\draw[color=black] (-7.56341445377412,0.6006720161394983) node {$W_s$};
	\draw[color=black] (-4.5,-2) node {$W_0$};
	\draw[color=black] (-5.52681364353465,2.3) node {$W$};
	\end{scriptsize}
	\end{tikzpicture}
	\caption{The $(k;k-t,k-t+1)$-SPIDs described in Class III.}
	\label{fig:example3}
\end{figure}
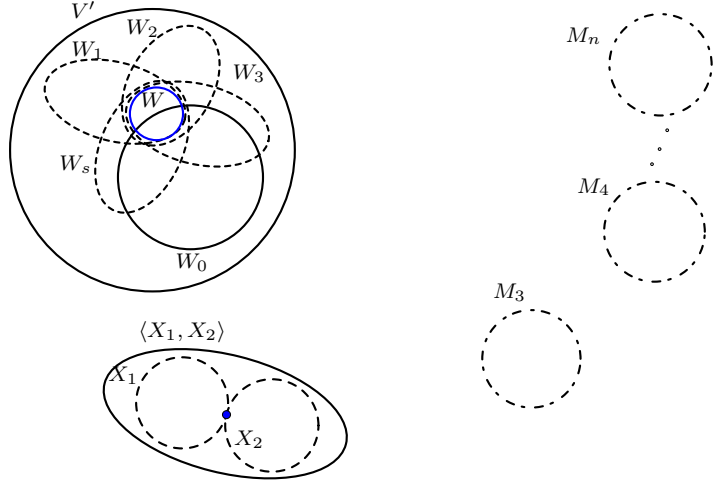

Examples in this class may contain $(k-t)$-sunflowers not of maximal dimension and $(k-t+1)$-sunflowers of maximal dimension. 

\bigskip

\noindent {\bf $\bullet$ \, Class IV}
\vspace{0.2cm}

\noindent Choose integers $n\geq 3$, $2\leq s<n$ and $k,t$ such that $2 \leq t \leq k-1$. Let $\mathbb{V}$ be a vector space over a field $\mathbb{F}$ which is either infinite or else a finite field $\mathbb{F}$ of order $q$ with $q$ a prime power such that $\frac{q^{k-t+2}-1}{q-1} \geq s+2$. Let $V'$, $M_1,\ldots,M_n$ be linearly independent subspaces of $\mathbb{V}$ such that $\dim V'=k-t+2$, and $\dim M_i=t-1$, for $i=1,\ldots,n$.\\
Let $V_0,W_0,W_1,\ldots,W_s$ be $s+2$  $(k-t+1)$-spaces in $V'$ such that they do not go through the same $(k-t)$-space, with $W_1,\ldots,W_s$ distinct (Figure \ref{fig:exemple4}) (which in the case $\mathbb{V}$ is a vector space over a finite field of order $q$, exist for the above assumption on $q$). We define the sets

$$\pi_1  = \spanset{V_0,M_1}, \, \pi_2=\spanset{W_0,M_2},$$ 
$$\pi_{3}=\spanset{W_1,M_{3}}, \ldots,\, \pi_{m_1}=\spanset{W_1,M_{m_1}},$$ 
$$\pi_{m_1+1}=\spanset{W_2,M_{m_1+1}}, \ldots,\,  \pi_{m_2}=\spanset{W_2,M_{m_2}},$$
$$ \ldots$$
$$\pi_{m_{s-1}+1}=\spanset{W_s,M_{m_{s-1}+1}},\ldots,\, \pi_{n}=\spanset{W_s,M_n}.$$

Clearly, the set $\cS$ is a $\{k;k-t,k-t+1\}$-SPID such that it is not a $(k-t)$-junta and
\begin{equation*}
\spanset{\cS}=\spanset{\pi_1,\ldots,\pi_n}=\spanset{V',M_1,M_2,M_3,\ldots,M_{n}}.
\end{equation*}
Since $V'$, $M_1,M_2,M_3,\ldots,M_n$ are linearly independent, we find that
\begin{align*}
\dim\spanset{\cS}& = k-t+2 +n\cdot(t-1) \\
& = k+ (n-1)(t-1)+1,
\end{align*}
and $\delta'(\cS)=(k,t,t-1, \ldots, t-1).$

\newpage
 
 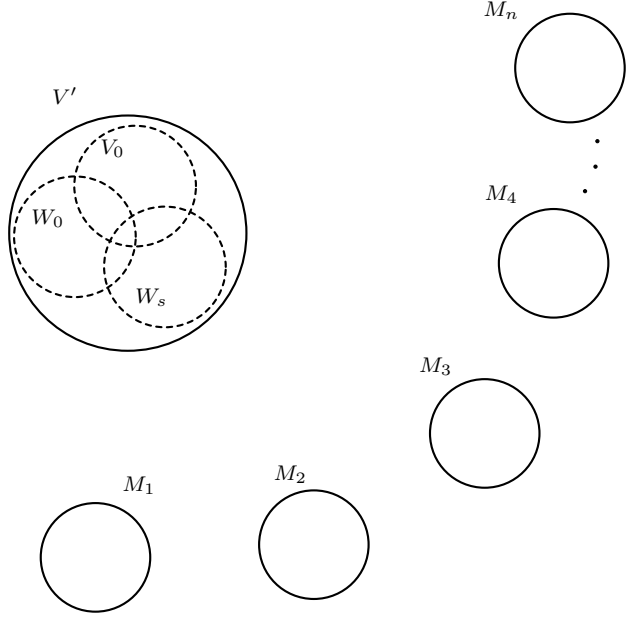
\begin{figure}[h!]
 	\centering
 	\begin{tikzpicture}[line cap=round,line join=round,>=triangle 45,x=1cm,y=1cm, scale=0.4]
 	\clip(-17.013897211351335,-9.742201126884343) rectangle (26.41558315826043,11.740499345571877);
 	\draw [line width=0.8pt] (-4,3) circle (3.8997948664000264cm);
 	\draw [line width=0.8pt, ,dash pattern=on 2pt off 2pt] (-3.76,4.56) circle (2cm);
 	\draw [line width=0.8pt, ,dash pattern=on 2pt off 2pt] (-2.78,1.88) circle (2cm);
 	\draw [line width=0.8pt, ,dash pattern=on 2pt off 2pt] (-5.74,2.88) circle (2cm);
 	\draw [line width=0.8pt] (2.1167356963309496,-7.327884763059891) circle (1.8cm);
 	\draw [line width=0.8pt] (7.735254928017719,-3.640042081949059) circle (1.8cm);
 	\draw [line width=0.8pt] (-5.063143959038027,-7.748724101822344) circle (1.8cm);
 	\draw [line width=0.8pt] (10,2) circle (1.8cm);
 	\draw [line width=0.8pt] (10.53717839301095,8.473581330058778) circle (1.8cm);
 	\begin{scriptsize}
 	\draw[color=black] (-6.069077424486191,7.565109929074949) node {$V'$};
 	\draw[color=black] (-4.5,5.8) node {$V_0$};
 	\draw[color=black] (-3.2239339085767185,0.874332869683477) node {$W_s$};
 	\draw[color=black] (-6.631853064995758,3.5) node {$W_0$};
 	\draw[color=black] (1.3408018422230954,-5) node {$M_2$};
 	\draw[color=black] (6.124394786554407,-1.4) node {$M_3$};
 	\draw[color=black] (-3.6303829822780718,-5.347464489283359) node {$M_1$};
 	\draw[color=black] (8.281701408507745,4.282252026102498) node {$M_4$};
 	\draw[color=black] (8.281701408507745,10.378988131622764) node {$M_n$};
 	\draw [fill=black] (11.03304898433229,4.391680622868242) circle (1.5pt);
 	
 	\draw [fill=black] (11.376967431310357,5.2045787702709445) circle (1.5pt);
 	
 	\draw [fill=black] (11.439498058033642,6.048742231035289) circle (1.5pt);
 	
 	\end{scriptsize}
 	\end{tikzpicture}
 	\caption{The $(k;k-t,k-t+1)$-SPIDs described in Class IV.}
 	\label{fig:exemple4}
 \end{figure}

Examples in this last class may contain $(k-t+1)$-sunflowers of maximal dimension and $(k-t)$-sunflowers not of maximal dimension.

\section{Tightness of the junta-property bound for $(k;k-t,k-t+1)$-SPIDs}

We will prove the following classification result.

\begin{theorem}
Let $\cS$ be a $(k;k-t,k-t+1)$-SPID in a vector space $\V$, with $| \cS |= n \geq 3$ and $2 \leq  t \leq k-1$. If the dimension of $\spanset{\cS}$ is  $k+(n-1)(t-1)+1$, then $\cS$ is either a $(k-t)$-junta or $\cS$ is one of the examples described in Classes I, II, III or IV.
\end{theorem}

First, we state the following lemma.

\begin{lemma} \label{no3petals}
Let $\cS$ be a $(k;k-t,k-t+1)$-SPID ($2 \leq  t \leq k-1$) in a vector space $\mathbb{V}$ such that $n=| \cS | \geq 3$ and $\cS$ is not a $(k-t)$-junta.\\
If $\dim\spanset{\cS}=k+(n-1)(t-1)+1$ and there is not a $(k-t)$-sunflower of maximal dimension with at least three petals in $\cS$, then $\cS$ is equivalent to one of the examples described in Classes $II$, $III$, or $IV$.
\end{lemma}

\begin{proof}
By Propositions \ref{sorting} and  \ref{turningspaces}, we may sort $k$-subspaces in $\cS$ in such a way that 
\begin{equation}\label{eq:parametersnosunflower}
    (\delta_1,\delta_2, \ldots,\delta_n)=(k,t,t-1,\ldots,t-1).
\end{equation}
Also, arguing as in the proof of Proposition $\ref{turningspaces}$, we get that $\dim(\pi_4 \cap \spanset{\pi_1,\pi_2,\pi_3})=\dim(\pi_4 \cap \spanset{\pi_1,\pi_2})=k-t+1$, and hence
\begin{equation*}
    \pi_4 \cap \pi_3 \subseteq \pi_4 \cap \spanset{\pi_1,\pi_2,\pi_3}=\pi_4 \cap \spanset{   \pi_1,\pi_2} \subseteq \spanset{\pi_1,\pi_2}.
\end{equation*}
Eventually rearranging the spaces $\pi_3,\ldots,\pi_n$ in $\cS$, we can repeat the previous argument, getting
\begin{equation*}
    \pi_i \cap \pi_j \subseteq \spanset{\pi_1,\pi_2},
\end{equation*}
for all distinct $\pi_i$ and $\pi_j$, with $i,j \in \{3,\ldots,n\}$. Moreover, $\dim(\pi_i \cap  \spanset{\pi_1,\pi_2})=\dim(\pi_j \cap \spanset{\pi_1, \pi_2})=k-t+1$.

Now, define in  $\cS'=\{\pi_3,\ldots,\pi_n\}$, the following binary relation
\begin{equation*}
\pi_i \sim \pi_j \Longleftrightarrow \pi_i \cap \spanset{\pi_1,\pi_2}= \pi_j \cap \spanset{\pi_1,\pi_2},
\end{equation*}
for $i,j=3,\ldots,n$.\\
Clearly, $\sim$ is an equivalence relation on $\cS'$. The $k$-spaces of an equivalence class meet  $\spanset{\pi_1,\pi_2}$ in the same $(k-t+1)$-space. In this way, we get $s$ distinct $(k-t+1)$-dimensional spaces in $\spanset{\pi_1,\pi_2}$, say $W_1,\ldots,W_s$, where $s$ is a given integer, $1 \leq s \leq n-2$, pairwise intersecting in a $(k-t)$-space. Indeed, let $\pi_i$ and $\pi_j$ be $k$-spaces of $\cS'$ in different equivalence classes. Then, by the proof of Proposition \ref{turningspaces}, $\dim(\pi_i \cap  \spanset{\pi_1,\pi_2})=\dim(\pi_j \cap \spanset{\pi_1, \pi_2})=k-t+1$,
$$\pi_i \cap  \spanset{\pi_1,\pi_2}=W_\ell \,\,\textnormal{and}\,\,
\pi_j \cap  \spanset{\pi_1,\pi_2} =W_m,$$
for some distinct $\ell,m \in \{1, \ldots,s\}$. Hence,
$$\pi_i \cap \pi_j= \pi_i \cap \pi_j \cap  \spanset{\pi_1,\pi_2}=W_\ell \cap W_m.$$
Since $k-t \leq \dim(\pi_i \cap \pi_j)=\dim(W_\ell \cap W_m)$ and $W_\ell,W_m$ are distinct $(k-t+1)$-subspaces,
$$\dim(W_\ell \cap W_m)=k-t.$$

Now, since the relation $\sim$ induces a partition $J_1,J_2,\ldots,J_s$ on the elements of the index set $\{3,4,\ldots,n\}$, by Proposition \ref{turningspaces}, we can label appropriately  the elements of $\cS$, obtaining
\begin{equation*}
\begin{split}
    \pi_{j_1} & =\spanset{W_1,M_{j_1}} \hspace{2cm} \textnormal{with} \,\, j_1 \in J_1, \\
    & \hspace{-0.5cm} \pi_{j_2}=\spanset{W_2,M_{j_2}}  \hspace{2cm} \textnormal{with}  \,\, j_2 \in J_2, \\
   &  \ldots \\
   & \hspace{-0.5cm} \pi_{j_s}=\spanset{W_s,M_{j_s}} \hspace{2cm} \textnormal{with}  \,\, j_s \in J_s,
\end{split}
\end{equation*}
where the elements in the set $\{M_{j_{h}} : j_{h}\in J_h,\,\,h\in \{1,2,\ldots,s\}\}$, are certain linearly independent $(t-1)$-spaces in $\mathbb{V}$. 

We divide the remainder of the proof in two steps:
\begin{itemize}
    \item [1)] First, we look at the case where all elements in $\cS'$ meet $\spanset{\pi_1,\pi_2}$ in the same $(k-t+1)$-space, say $W$. It is clear that for $3 \leq j \leq n$ and $i=1,2$, $\pi_i \cap \pi_j=\pi_i \cap W$. Indeed, since $W = \pi_j \cap \spanset{\pi_1, \pi_2}$; we have
    	$$\pi_i \cap \pi_j \subseteq W.$$
    	Hence, $\pi_i \cap \pi_j \subseteq W \cap \pi_i$. On the other hand, since 
    $$W = \pi_j \cap \spanset{\pi_i ,\pi_2} \subseteq \pi_j,$$
    then we also have $W \cap \pi_i \subseteq \pi_j \cap \pi_i$.
Next we show that
    \begin{equation*}
        \dim(\pi_i \cap \pi_j)=k-t,
    \end{equation*}
for $3 \leq j \leq n$ and $i=1,2$. To this aim, suppose that either the space $\pi_1$ or $\pi_2$ contains $W$ ($W \not \subseteq \pi_1 \cap \pi_2$, since $\dim(\pi_1 \cap \pi_2)=k-t$). For instance, let $\pi_1$ contain $W$. Then, $\pi_2 \cap W=\pi_1 \cap \pi_2$; in fact, we have that $\pi_1 \cap \pi_2 \supseteq \pi_2 \cap W = \pi_2 \cap \pi_j$ with $j \in \{3,...,n\}$. But then $\cS$ is a $(k-t)$-junta; a contradiction. 

Hence, $\pi_1 \cap W=W_1$ and $\pi_2 \cap W=W_2$ are $(k-t)$-spaces, and they are distinct otherwise $\cS$ is again a $(k-t)$-junta. Precisely, they are two hyperplanes of $W$. We denote by $W'$ the $(k-t-1)$-space of $W$ in which they meet, and choose a basis of $\mathbb{V}$ in such a way that the following happens
    \begin{equation*}
        \pi_1 \cap W=\spanset{W',a_1} \,\,\, \textnormal{and} \,\,\, \pi_2 \cap W=\spanset{W',a_2},
    \end{equation*}
    with $a_1,a_2$ distinct $1$-spaces in $W_1 \setminus W'$ and $W_2 \setminus W'$, with $W', a_1,a_2$ linearly independent. Then, there also exist two $t$-spaces $X_1$ and $X_2$, having a $1$-space in common and such that
    \begin{equation*}
        \pi_1=\spanset{W',a_1,X_1}, \,\,\,\,\,\,\, \pi_2=\spanset{W',a_2,X_2}.
    \end{equation*}
This finally means that $\cS$ is one of the examples in Class II.

    \item[2)] Now, we suppose that $s \geq 2$. In this case, $W_1, \ldots, W_s$ are $(k-t+1)$-spaces pairwise intersecting in a $(k-t)$-space. Hence, by \cite[Section 9.3]{brouwer}, either 
    \begin{itemize}
        \item [$(a)$] they have a $(k-t)$-space in common, or
        \item[$(b)$] they lie in a $(k-t+2)$-space $V'$.
    \end{itemize} 

Note explicitly that for $s=2$, $(a)$ and $(b)$ are equivalent.
If  $s \geq 3$, we will show that 
\begin{equation}\label{inside1}
    \dim\spanset{W_1,W_2,\ldots,W_s}=k-t+2,
\end{equation}
which is equivalent to prove that, for all $1\leq h \leq s,$
 \begin{equation}\label{inside2}
     W_h \subseteq \spanset{W_1,W_2}.
 \end{equation}
 
Suppose that $W_1,W_2,\ldots,W_s$ go through a $(k-t)$-space in $\spanset{\pi_1,\pi_2}$ and let $\pi_{j_1}$,  $ \pi_{j_2},\pi_{j_h}$ be $k$-spaces  belonging to different equivalence classes with respect to $\sim$, such that
\begin{equation*}
\pi_{j_1}  =\spanset{W_1,M_{j_1}}  \hspace{0.5cm}
 \pi_{j_2}=\spanset{W_2,M_{j_2}}
 \hspace{0.5cm}
 \pi_{j_h}=\spanset{W_h,M_{j_h}}.
\end{equation*}

Since there is not a sunflower of maximal dimension with at least three petals contained in $\cS$, we have 
\begin{equation*}
    \dim(\pi_{j_h} \cap \spanset{\pi_{j_1},\pi_{j_2}})\geq k-t+1.
\end{equation*}
Then, by applying Grassmann's Formula, we obtain
\begin{equation*}
\begin{split}
   k-t+1  & \leq \dim(\pi_{j_h} \cap \spanset{\pi_{j_1},\pi_{j_2}})=2k+t -\dim\spanset{W_1,W_2,W_h,M_{j_1},M_{j_2},M_{j_h}} \\
   & = 2k+t- 3(t-1)- (\dim W_h + \dim \spanset{W_1,W_2} - \dim(W_h \cap \spanset{W_1,W_2}).
\end{split}
\end{equation*}
This implies $\dim(W_h \cap \spanset{W_1,W_2}) \geq k-t+1$ and hence we get property  $(\ref{inside2})$.\\
So, all $(k-t+1)$-spaces $W_1,\ldots,W_s$ lie in a $(k-t+2)$-space, say $V'$.\\
Obviously, $V'$ is contained in $\spanset{\pi_1,\pi_2}$ and
\begin{equation}\label{inequality1}
  \hspace{1cm}  k-t \leq \dim(\pi_i \cap V') \leq k-t+1, \hspace{2cm} \textnormal{for}\,\,\, i=1,2.
\end{equation}
Indeed, since for $i \in \{1,2\}$ and for any $j \in \{3,\ldots,n\}$,
	$$\pi_i \cap \pi_j= \pi_i \cap  \pi_j \cap \spanset{\pi_1,\pi_2}=\pi_i \cap W_h \subseteq \pi_i \cap V',$$
for some $h\in \{1,\ldots,s\}$, then the first inequality in (\ref{inequality1}) follows.\\
On the other hand, if $\dim(\pi_i \cap V') \geq k-t+2$, for $i=1$ or $2$, then $V'$ is contained either in $\pi_1$ or in $\pi_2$ (not in both since $\dim(\pi_1 \cap \pi_2)=k-t$). Without loss of generality, we can suppose that $V'$ is contained in $\pi_1$. Then  $$ \pi_2 \cap \pi_{j_h}=\pi_2 \cap W_h \subseteq \pi_2 \cap V' \subseteq \pi_1 \cap \pi_2,$$ 
for $h \in \{1,\ldots,s\}.$ This implies that $\pi_1 \cap \pi_2$ is contained in all elements of $\cS$ and hence it is a $(k-t)$-junta.

Furthermore, $\pi_1 \cap V'$ and $\pi_2 \cap V'$ are distinct subspaces. Indeed,
\begin{itemize}
\item [$(\diamond)$] if $\pi_1 \cap V'=\pi_2 \cap V'$ and it is a $(k-t+1)$-space, then $\pi_1 \cap \pi_2$ is a $(k-t+1)$-space, a contradiction;
\item[$(\diamond \diamond)$] if $\pi_1 \cap V'=\pi_2 \cap V'$ is a $(k-t)$-space, since for $i=1,2$ and $h=1,\ldots,s$, $\pi_i \cap W_h$ has dimension at least $k-t$ and $W_h \subseteq V'$, we have that
 $$\pi_1 \cap W_h= \pi_1 \cap V'=\pi_2 \cap V' =\pi_2 \cap W_h.$$
This implies that $\pi_1 \cap \pi_2$ is contained in all elements of $\cS$; again this is not the case.
\end{itemize} 

Now, let $W_h$ be a $(k-t+1)$-space with $1 \leq h \leq s $, then
\begin{equation}\label{inequality2}
\begin{split}
k-t = & \dim(\pi_1 \cap \pi_2) \geq \dim(\pi_1 \cap \pi_2 \cap W_h) \geq \\
&  \dim(\pi_1 \cap W_h)+ \dim(\pi_2 \cap W_h)-\dim W_h \geq 2(k-t)-k+t-1=k-t-1.  
\end{split}
\end{equation}
By taking into account Inequalities (\ref{inequality1}) and (\ref{inequality2}), 
the discussion may be reduced to one of the following three cases:
\begin{itemize}
    \item [$(i)$] $\dim(\pi_1 \cap V')=\dim(\pi_2 \cap V')=k-t $\,\, (and \,\,$\dim(\pi_1 \cap \pi_2 \cap V')=k-t-1$).
    \item [$(ii)$] $\pi_1$ and $\pi_2$ meet $V'$ in subspaces with different dimensions.
    \item [$(iii)$] $\pi_1 \cap V'$ and $\pi_2 \cap V'$ are two hyperplanes of $V'$.
\end{itemize}

Case $(i)$: We shall show that for all $3 \leq j \leq n$,
\begin{equation}\label{always_the_same_space}
    \pi_j \cap \spanset{\pi_1,\pi_2}= \spanset{\pi_1 \cap V', \pi_2 \cap V'}.
\end{equation}
Since $\pi_j \cap \spanset{\pi_1,\pi_2} \subseteq V'$, and $\pi_1$ and $\pi_2$ meet $V'$ in a $(k-t)$-space,

\begin{center}
 $\pi_j \cap \pi_1=\pi_1 \cap V'$ and $\pi_j \cap \pi_2=\pi_2 \cap V'$ ,  
\end{center}
obtaining that
\begin{equation*}
    \pi_j \cap \spanset{\pi_1,\pi_2} \supseteq \spanset{\pi_j \cap \pi_1, \pi_j \cap \pi_2}=\spanset{\pi_1 \cap V', \pi_2 \cap V'}.
\end{equation*}
However, since they both are $(k-t)$-spaces in $V'$, we obtain equality stated in (\ref{always_the_same_space}). Hence, every $\pi_j$, $j=3,\ldots,n$, meets $\spanset{\pi_{1},\pi_2}$ in the same $(k-t+1)$-subspace. But this contradicts $s \geq 2$.\\
\noindent Case $(ii)$:  We can suppose, without loss of generality, that 
\begin{center}
$\dim(\pi_1 \cap V')=k-t$ \,\,\, and \,\,\,  $ \dim(\pi_2 \cap V')=k-t+1$.
\end{center}
Clearly, $\pi_1 \cap V' \not \subseteq \pi_2 \cap V'$, otherwise 
$$\pi_1 \cap W_h =\pi_1 \cap V' \subseteq \pi_2 \cap V'.$$

This implies that $\pi_1 \cap \pi_2 \subseteq W_h$, for $h=1,\ldots,s$, and then $\cS$ is a $(k-t)$-junta with center $\pi_1 \cap \pi_2$.

Since $W=\pi_1 \cap \pi_2 \cap V'$ is a $(k-t-1)$-space, there exists a $t$-space $X_1$ contained in $\pi_1$ and a $(t-1)$-space $X_2$ contained in $\pi_2$ 
both disjoint from $V'$, for $i=1,2$, and such that $\spanset{X_1,X_2}=2t-2$. Then
$$\pi_1=\spanset{\pi_1 \cap V',X_1} \,\,\,\textnormal{and}\,\,\,\pi_2=\spanset{\pi_2\cap V',X_2}.$$
We note explicitly that
\begin{equation}\label{sunflower_remark}
\pi_1 \cap V' = \pi_1 \cap W_h \subseteq W_h,
\end{equation}
for $h\in \{1,\ldots,s\}$, since $\dim(\pi_1 \cap V')=k-t$.\\
Case $(iii)$: Now, we suppose that $\pi_1 \cap V'$ and $\pi_2 \cap V'$ are hyperplanes of $V'$, say $V_0$ and $W_0$, respectively. Then, there exists $X_i$, $i=1,2$, a $(t-1)$-space in $\pi_i$ disjoint from $V'$, such that 

$$\pi_1=\spanset{V_0,X_1} \,\,\,\textnormal{and}\,\,\,\pi_2=\spanset{W_0,X_2}.$$

Again, by Grassmann's Formula, we obtain that $X_1,X_2,V'$ are linearly independent and $\dim\spanset{X_1,X_2}=2t-2$.
\end{itemize}
So, the discussion in Case $(ii)$ provides us with an example described in Class III, while Case $(iii)$ gives an example in Class IV.
\end{proof}

\begin{remark} \textnormal{Let  $\mathcal{W}= \{W_1,\ldots,W_s,\pi_2 \cap V' \}$ be the set of $(k-t+1)$-spaces in $V'$ with $2 \leq s \leq n-3$.\\
In Case $(ii)$, if $s \geq 3$ by Formula (\ref{sunflower_remark}), the first $s$ subspaces in $\mathcal{W}$ form a sunflower with center $\pi_1 \cap V'$, and $\pi_2 \cap V'$ not through $\pi_1 \cap V'$.\\
In Case $(iii)$, considering $\pi_1 \cap V'$ and $\pi_2\cap V'$, one of them or both could be in $\{W_1,\ldots,W_s\}$.\\
If $s=2$, at most one of $\pi_1 \cap V'$ and $\pi_2\cap V'$ can coincide with $W_1$ or $W_2$. Otherwise, $W_1 \cap W_2=\pi_1 \cap \pi_2$ and it is contained in all elements of $\cS$.\\
In particular, if $s=2$ and $n=4$, it is straightforward to see that exactly one of $\pi_1 \cap V'$ and $\pi_2 \cap V'$ must necessarily be equal to $W_1$ or $W_2$.}
\end{remark}

We are now in the position to prove the main result of this section.

\medskip
\noindent {\bf Proof of Theorem 3.1}
\medskip

\noindent We assume that $\cS$ is not a $(k-t)$-junta and denote the elements of $\cS$ by $\pi_1,\pi_2, \ldots, \pi_n$. We will consider all possible orderings of the spaces in $\cS$ such that the parameters $(\delta_2,\ldots,\delta_n)$ are non-increasing.\\
Since $\dim\spanset{\cS}=k+(n-1)(t-1)+1$, we have the equality in (\ref{bound}) of Theorem \ref{initial}.\\ Hence, if $m \geq 3$, we have
\begin{equation} \label{case1}
    (\delta_2, \ldots, \delta_n)=(\underbrace{t, \ldots, t}_{m-1 \,\,\text{times}}, \underbrace{t-1, \ldots, t-1}_{n-m-1\,\, \text{times}}, t+1-m),
\end{equation}
otherwise $m=2$ and we have
\begin{equation}\label{case2}
    (\delta_2, \ldots, \delta_n)=(t, t-1, \ldots,t-1).
\end{equation}
\begin{itemize}
    \item [$\circ$] Suppose that we can find a permutation of the elements in $\mathcal{S}$ such that $\delta(\cS)$ is as in (\ref{case1}), for $m \geq 3$. Then, by Lemma \ref{3petals}, it follows that $\mathcal{S}$ belongs to Class I.

\item[$\circ$] If otherwise $(\delta_2, \ldots, \delta_n)$ is as in $(\ref{case2})$, by Proposition \ref{turningspaces}, there is no $(k-t)$-sunflower of maximal dimension with at least three petals. The result then follows by Lemma \ref{no3petals}. 
\end{itemize}
\vspace{-0.5cm}
\hspace{14.7cm}\qedsymbol
\\

\noindent The authors thank the referees for careful reading of the paper and for their valuable suggestions which improved the article.

\vspace{2cm}

\noindent Giovanni Longobardi,\\
Department of Management and Engineering,\\
University of Padua,\\
Stradella S. Nicola, 3, 36100 Vicenza, Italy,\\
\emph{giovanni.longobardi@unipd.it}

\bigskip
\noindent Leo Storme,\\
Department of Mathematics: Analysis, Logic and Discrete Mathematics,\\
Ghent University,\\
Krijgslaan 281, Building S8, 9000 Gent, Flanders, Belgium,\\
\emph{leo.storme@ugent.be}

\bigskip 
\noindent Rocco Trombetti,\\
Department of Mathematics and Applications “R. Caccioppoli”,\\
University of Naples ”Federico II”,\\
via Vicinale Cupa Cintia, I-80126 Napoli, Italy,\\
\emph{rtrombet@unina.it}

\end{document}